\documentclass[pra]{revtex4}
 \usepackage{amssymb} \usepackage{graphicx}
\usepackage[utf8]{inputenc}
\usepackage[english]{babel} 
\usepackage{amsthm}

\newtheorem{definition}{Definition}
\newtheorem{proposition}{Proposition}
\newtheorem{theorem}{Theorem}
\newtheorem{corollary}{Corollary}

\begin{document}
 \title{Bilinear forms of K\"{a}hlerian twistor spinors}

\author{\"Umit Ertem}
 \email{umit.ertem@tedu.edu.tr, umitertemm@gmail.com}
\address{TED University, Ziya G\"{o}kalp Caddesi, No:48, 06420, Kolej \c{C}ankaya, Ankara,
Turkey\\}

\begin{abstract}

Generalization of twistor spinors to K\"{a}hler manifolds which are called K\"{a}hlerian twistor spinors are considered. We find the differential equation satisfied by the bilinear forms of K\"{a}hlerian twistor spinors. We show that the bilinear form equation reduces to K\"{a}hlerian conformal Killing-Yano equation under special conditions. We also investigate the special cases of holomorphic and anti-holomorphic K\"{a}hlerian twistor spinors and the bilinear forms of alternative definitions for K\"{a}hlerian twistor spinors.

\end{abstract}

\maketitle

\section{Introduction}

On an $n$-dimensional Riemannian spin manifold $M^n$, a special class of spinors which are in the kernel of the twistor operator can exist and these are called twistor spinors \cite{Lichnerowicz1, Lichnerowicz2, Habermann1, Baum Friedrich Grunewald Kath, Baum Leitner, Ertem1}. They are solutions of the twistor equation
\[
\nabla_X\psi=\frac{1}{n}\widetilde{X}.\displaystyle{\not}D\psi
\]
for any vector field $X$ and its metric dual $\widetilde{X}$. Here, $\psi$ is the spinor, $\displaystyle{\not}D$ is the Dirac operator and $.$ denotes the Clifford multiplication. Because of the conformal covariance of the twistor equation, twistor spinors are related to the conformal symmetries of the manifold \cite{Lichnerowicz2, Ertem1, Ertem3, Ertem4}. On the other hand, K\"{a}hler spin manifolds do not admit nontrivial twistor spinors \cite{Hijazi1, Kirchberg1}. However, one can define a generalization of twistor spinors to K\"{a}hler spin manifolds. By considering the decomposition of the spinor bundle into the eigenbundles of the K\"{a}hler form on a K\"{a}hler manifold, one can define a K\"{a}hlerian twistor operator and the spinors in the kernel of this K\"{a}hlerian twistor operator are called K\"{a}hlerian twistor spinors \cite{Pilca}. K\"{a}hlerian twistor equation for a spinor $\psi$ of type $r$ on a K\"{a}hler manifold $M^{2m}$ is written as
\[
\nabla_X\psi=\frac{m+2}{8(r+1)(m-r+1)}(\widetilde{X}.\displaystyle{\not}D\psi+J\widetilde{X}.\displaystyle{\not}D^c\psi)+\frac{m-2r}{8(r+1)(m-r+1)}i(J\widetilde{X}.\displaystyle{\not}D\psi-\widetilde{X}.\displaystyle{\not}D^c\psi).
\]
where $J$ is the complex structure and $\displaystyle{\not}D^c$ is the conjugate Dirac operator defined in (20). The manifolds admitting nontrivial K\"{a}hlerian twistor spinors are investigated in \cite{Pilca}. Special types of K\"{a}hlerian twistor spinors are also defined previously in \cite{Kirchberg2, Hijazi2}. Moreover, K\"{a}hlerian Killing spinors as the limiting cases of the eigenvalues of the Dirac operator on compact K\"{a}hler manifolds can also be defined as the special cases of K\"{a}hlerian twistor spinors \cite{Kirchberg3, Moroianu1}

From a spin invariant inner product defined on the space of spinors, one can define a squaring map which is used in the construction of differential forms from spinors. These different degree differential forms are called bilinear forms of spinors. For a twistor spinor on a Riemannian manifold, the bilinear forms are related to the conformal hidden symmetries of the manifold. Bilinear forms of twistor spinors correspond to conformal Killing-Yano (CKY) forms which are antisymmetric generalizations of conformal Killing vector fields to higher degree differential forms \cite{Acik Ertem, Ertem1, Ertem2}. So, the bilinear $p$-forms $\omega$ of a twistor spinor satisfy the following CKY equation on $M^n$
\[
\nabla_X\omega=\frac{1}{p+1}i_Xd\omega-\frac{1}{n-p+1}\widetilde{X}\wedge\delta\omega
\]
where $d$ and $\delta$ are exterior derivative and co-derivative operators, $i_X$ is the interior derivative or contraction operator with respect to $X$ and $\wedge$ denotes the wedge product. In this paper, we investigate the construction of bilinear forms for K\"{a}hlerian twistor spinors. We show that the bilinear forms of K\"{a}hlerian twistor spinors satisfy a more general differential equation given in Theorem 1 which can reduce to the K\"{a}hlerian CKY equation in some special cases. Properties of K\"{a}hlerian CKY forms are studied in \cite{Moroianu Semmelmann}. Moreover, we also consider the bilinear forms of holomorphic and anti-holomorphic K\"{a}hlerian twistor spinors and also the special K\"{a}hlerian twistor spinor definitions of \cite{Kirchberg2} and \cite{Hijazi2}.

The paper is organized as follows. In Section II, we give the basic definitions and operators on K\"{a}hler manifolds. One can find more details on K\"{a}hler manifolds in \cite{Moroianu2}. Section III includes the spin geometry of K\"{a}hler manifolds and the definitions of K\"{a}hlerian twistor spinors. In Section IV, we construct the bilinear forms of Riemannian and K\"{a}hlerian twistor spinors and prove the theorem for the bilinear form equation of K\"{a}hlerian twistor spinors. Section V includes the special cases of holomorphic and anti-holomorphic K\"{a}hlerian twistor spinors and the special definitions of \cite{Kirchberg2} and \cite{Hijazi2}.

\section{Operators on K\"{a}hler manifolds}

Let us consider a $n=2m$ dimensional K\"{a}hler manifold $(M^{2m}, g, J)$ with a Riemannian metric $g$ and a complex structure $J$ satisfying $J^2=-I$ where $I$ is the identity. Hence, $J$ is integrable and satisfies $\nabla_XJ=0$ for any vector field $X$, namely the Nijenhuis tensor $N^J$ of $J$ vanishes. $g$ and $J$ have the relation $g(JX, JY)=g(X,Y)$ for any vector fields $X$ and $Y$.

The complex structure $J$ induces a splitting on the complexified tangent bundle $TM^{\mathbb{C}}=TM\otimes_{\mathbb{R}}\mathbb{C}$;
\[
TM^{\mathbb{C}}=TM^+\oplus TM^-.
\]
Here, $TM^+$ corresponds to $+i$ eigenspace of $J$ and $TM^-$ corresponds to $-i$ eigenspace of $J$. For any $X\in TM^{\mathbb{C}}$, we have $X=X^++X^-$ and $X^+\in TM^+$ and $X^-\in TM^-$ are defined as follows
\[
X^+=\frac{1}{2}(X-iJX)\quad\quad , \quad\quad X^-=\frac{1}{2}(X+iJX).
\]
Then, we can define an orthonormal basis $\{X_a, JX_a\}$ for $TM^{\mathbb{C}}$ with $a=1,2,...,m$. Similarly, the above splitting also induce a splitting on the cotangent bundle $T^*M^{\mathbb{C}}=T^*M\otimes_{\mathbb{R}}\mathbb{C}$ and we can define an orthonormal basis $\{e^a, Je^a\}$ on the cotangent bundle with $a=1,2,...,m$. We have the following relations between the frame and coframe basis;
\begin{eqnarray}
e^a(X_b)&=&g(X^a, X_b)=\delta^a_b\nonumber\\
Je^a(JX_b)&=&g(JX^a, JX_b)=g(X^a, X_b)=\delta^a_b\nonumber\\
e^a(JX_b)&=&g(X^a, JX_b)=-g(JX^a, X_b)\nonumber\\
Je^a(X_b)&=&g(JX^a, X_b)=-g(X^a, JX_b)\nonumber
\end{eqnarray}
where $\delta^a_b$ denotes the Kronecker delta.

The above splitting of tangent and cotangent bundles can also be generalized to complex tensor bundles on $M$. In particular, for the exterior bundle $\Lambda M^{\mathbb{C}}=\Lambda M\otimes_{\mathbb{R}}\mathbb{C}$, we have the following splitting;
\[
\Lambda^rM^{\mathbb{C}}=\bigoplus_{p+q=r}\Lambda^pM^+\otimes\Lambda^qM^-.
\]
The elements of $\Lambda^{p,q}M^{\mathbb{C}}:=\Lambda^pM^+\otimes\Lambda^qM^-$ are called as the forms of type $(p,q)$. So, a complex $r$-form is a sum of $(p,q)$-forms with $p+q=r$.

A closed (1,1)-form $\Omega$ which is written in terms of coframe basis as 
\begin{equation}
\Omega=\frac{1}{2}e^a\wedge Je_a
\end{equation}
is called the K\"{a}hler 2-form of $M$ and satisfies $d\Omega=0$ where $d$ is the exterior derivative operator. It can also be written in terms of the metric as $\Omega(X,Y):=g(JX,Y)$ for any vector fields $X$ and $Y$. The exterior derivative of a $(p,q)$-form is written as a sum of $(p+1, q)$ and $(p, q+1)$-forms. So, the exterior derivative operator $d$ acts as
\[
d:\Lambda^{p,q}M^{\mathbb{C}}\longrightarrow\Lambda^{p+1,q}M^{\mathbb{C}}\oplus\Lambda^{p,q+1}M^{\mathbb{C}}.
\]
Similarly, the coderivative operator $\delta$ acts as
\[
\delta:\Lambda^{p,q}M^{\mathbb{C}}\longrightarrow\Lambda^{p-1,q}M^{\mathbb{C}}\oplus\Lambda^{p,q-1}M^{\mathbb{C}}.
\]
On the other hand, the interior derivative operator $i_X$ with respect to a vector field $X$ and $i_{JX}$ with respect to a vector field $JX$ correspond to contractions of $(p,q)$-forms with respect to these vector fields and act as follows
\begin{eqnarray}
i_X:\Lambda^{p,q}M^{\mathbb{C}}\longrightarrow\Lambda^{p-1,q}M^{\mathbb{C}}\nonumber\\
i_{JX}:\Lambda^{p,q}M^{\mathbb{C}}\longrightarrow\Lambda^{p,q-1}M^{\mathbb{C}}\nonumber
\end{eqnarray}
and they satisfy the following relations;
\begin{equation}
i_{JX}=Ji_X\quad\quad , \quad\quad Ji_{JX}=-i_X.
\end{equation}
For zero torsion, the exterior derivative and coderivative operators can be written in terms of the covariant derivative $\nabla_X$ with respect to a vector field $X$ as follows;
\begin{eqnarray}
d&=&e^a\wedge\nabla_{X_a}\\
\delta&=&-i_{X^a}\nabla_{X_a}.
\end{eqnarray}

By using the K\"{a}hler form $\Omega$ given in (1), we can also define two operators acting on $(p,q)$-forms on $M$. These are the wedge product and contraction of a $(p,q)$-form $\alpha$ with $\Omega$ and defined as follows;
\begin{eqnarray}
L\alpha&:=&\Omega\wedge\alpha=\frac{1}{2}e^a\wedge Je_a\wedge\alpha\\
\Lambda\alpha&:=&i_{\Omega}\alpha=\frac{1}{2}i_{JX^a}i_{X_a}\alpha
\end{eqnarray}
and they act as
\begin{eqnarray}
L&:&\Lambda^{p,q}M^{\mathbb{C}}\longrightarrow\Lambda^{p+1,q+1}M^{\mathbb{C}}\nonumber\\
\Lambda&:&\Lambda^{p,q}M^{\mathbb{C}}\longrightarrow\Lambda^{p-1,q-1}M^{\mathbb{C}}.\nonumber
\end{eqnarray}
If a $(p,q)$-form $\alpha$ is in the kernel of the operator $\Lambda$, that is $\Lambda\alpha=0$, then $\alpha$ is called as a primitive form.

The complex structure $J$ can also be stated as a derivation on the space of $(p,q)$-forms acting as
\[
J:\Lambda^{p,q}M^{\mathbb{C}}\longrightarrow\Lambda^{p-1,q+1}M^{\mathbb{C}}.
\]
For any $\alpha\in\Lambda^{p,q}M^{\mathbb{C}}$, it can be written as
\begin{equation}
J\alpha:=Je^a\wedge i_{X_a}\alpha
\end{equation}
and it has the property
\begin{equation}
J(\alpha\wedge\beta)=J\alpha\wedge\beta+\alpha\wedge J\beta
\end{equation}
where $\beta\in\Lambda^{p,q}M^{\mathbb{C}}$. Moreover, $J$ commutes with both the operators $L$ and $\Lambda$;
\begin{equation}
[J,\Lambda]=0=[J,L].
\end{equation}

The operator $\Lambda$ commutes with the interior derivative operations $i_X$ and $i_{JX}$;
\begin{equation}
[i_X,\Lambda]=0\quad\quad , \quad\quad [i_{JX},\Lambda]=0
\end{equation}
and the commutation relations between the operator $L$ and the interior derivative operations are given by
\begin{equation}
[i_X, L]=J\widetilde{X}\wedge\quad\quad , \quad\quad [i_{JX}, L]=-\widetilde{X}\wedge
\end{equation}
where the 1-form $\widetilde{X}$ is the metric dual of the vector field $X$.

The generalizations of the operators $d$ and $\delta$ given in (3) and (4) to include the complex structure $J$ can also be constructed in the following way;
\begin{eqnarray}
d^c&=&Je^a\wedge\nabla_{X_a}\\
\delta^c&=&-i_{JX^a}\nabla_{X_a}
\end{eqnarray}
and the commutation relations between the operators $d$, $d^c$, $\delta$ and $\delta^c$ and the operators $L$ and $\Lambda$ are given as follows;
\begin{equation}
[L, d]=0=[L, d^c]\quad\quad , \quad\quad [\Lambda, \delta]=0=[\Lambda, \delta^c]
\end{equation}
\begin{equation}
[L, \delta]=d^c\quad , \quad [L, \delta^c]=-d\quad , \quad [\Lambda, d]=-\delta^c\quad , \quad [\Lambda, d^c]=\delta.
\end{equation}

\section{K\"{a}hlerian twistor spinors}

Let the K\"{a}hler manifold $M^{2m}$ admits a spin structure and $\Sigma M$ denotes the spinor bundle on $M$. The complex volume form on $M^{2m}$ is defined by
\begin{equation}
z^{\mathbb{C}}=i^m\prod_{i=1}^me_i.Je_i
\end{equation}
where $.$ denotes the Clifford multiplication. Since the dimension of $M$ is even, $z^{\mathbb{C}}$ has +1 and -1 eigenvalues and the spinor bundle $\Sigma M$ splits into a direct sum of spaces containing different eigenvalue eigenspinors of $z^{\mathbb{C}}$;
\[
\Sigma M=\Sigma^+ M\oplus\Sigma^- M.
\]
The K\"{a}hler form $\Omega$ given in (1) can be written in terms of Clifford multiplication as
\begin{equation}
\Omega=\frac{1}{2}e^a.Je_a
\end{equation}
and under the action of $\Omega$, the spinor bundle $\Sigma M$ splits into different eigenbundles of $\Omega$;
\begin{equation}
\Sigma M=\bigoplus_{r=0}^m\Sigma_rM
\end{equation}
where each $\Sigma_rM$ is the eigenbundle of $\Omega$ corresponding to the eigenvalue $i(2r-m)$. The rank of each eigenbundle is given by $\mathrm{rank}_{\mathbb{C}}(\Sigma_rM)={m \choose r}$. The elements of $\Sigma_rM$ are called K\"{a}hlerian spinors of type $r$. The eigenbundles of the complex volume form $z^{\mathbb{C}}$ also splits into the eigenbundles of $\Omega$ and we have
\begin{eqnarray}
\Sigma^+ M&=&\bigoplus_{0 \leq r \leq 2m}\Sigma_rM\quad,\quad\text{for $r$ even}\nonumber\\
\Sigma^- M&=&\bigoplus_{0 \leq r \leq 2m}\Sigma_rM\quad,\quad\text{for $r$ odd}\nonumber
\end{eqnarray}

The Levi-Civita connection $\nabla$ defined on $\Lambda M^{\mathbb{C}}$ can be induced on the spinor bundle $\Sigma M$ and it preserves the above splitting of the spinor bundle. The Dirac operator on $\Sigma M$ is defined as follows;
\begin{equation}
\displaystyle{\not}D:=e^a.\nabla_{X_a}
\end{equation}
and the conjugate Dirac operator is given by
\begin{equation}
\displaystyle{\not}D^c:=Je^a.\nabla_{X_a}.
\end{equation}
Both are square roots of the Laplacian and have the property $\displaystyle{\not}D^2=(\displaystyle{\not}D^c)^2$. Moreover, one can also define the operators
\begin{equation}
\displaystyle{\not}D^+=\frac{1}{2}(\displaystyle{\not}D-i\displaystyle{\not}D^c)\quad\quad , \quad\quad\displaystyle{\not}D^-=\frac{1}{2}(\displaystyle{\not}D+i\displaystyle{\not}D^c)
\end{equation}
and we have $\displaystyle{\not}D=\displaystyle{\not}D^+ +\displaystyle{\not}D^-$ and $\displaystyle{\not}D^c=i(\displaystyle{\not}D^+ -\displaystyle{\not}D^-)$. The Dirac operator $\displaystyle{\not}D$ acts on the subbundle $\Sigma_rM$ as
\[
\displaystyle{\not}D:\Sigma_rM\longrightarrow\Sigma_{r-1}M\oplus\Sigma_{r+1}M
\]
The metric duals of the vector fields $X^+$ and $X^-$ defined in Section II also act on the subbundle $\Sigma_rM$ and their action change the subbundles as follows;
\[
\widetilde{X}^+.\Sigma_rM\subseteq\Sigma_{r+1}M\quad\quad , \quad\quad \widetilde{X}^-.\Sigma_rM\subseteq\Sigma_{r-1}M.
\]
In general, the Clifford multiplication of a 1-form with an element of $\Sigma_rM$ gives an element of $\Sigma_{r-1}M\oplus\Sigma_{r+1}M$.

We define special types of spinors satisfying some differential equations in terms of $\nabla$, $\displaystyle{\not}D$ and $\displaystyle{\not}D^c$.
\begin{definition}
Let $M$ be a $n$-dimensional Riemannian spin manifold and $\psi\in\Sigma M$ is a spinor field on $M$. If $\psi$ satisfies the following equation for all vector fields $X\in TM$ and their metric duals $\widetilde{X}\in T^*M$
\begin{equation}
\nabla_X\psi=\frac{1}{n}\widetilde{X}.\displaystyle{\not}D\psi,
\end{equation}
then $\psi$ is called as a Riemannian twistor spinor.
\end{definition}
In the case of K\"{a}hler manifolds, the definition of twistor spinors can be generalized as follows;
\begin{definition}
Let ($M^{2m}$, $g$, $J$) be a $2m$-dimensional K\"{a}hler manifold. If $\psi\in\Sigma_rM$ is a spinor of type $r$ and satisfies the following equalities
\begin{eqnarray}
\nabla_{X^+}\psi&=&\frac{1}{2(m-r+1)}\widetilde{X}^+.\displaystyle{\not}D^-\psi\\
\nabla_{X^-}\psi&=&\frac{1}{2(r+1)}\widetilde{X}^-.\displaystyle{\not}D^+\psi
\end{eqnarray}
for all $X=X^+ + X^-\in TM^{\mathbb{C}}$, then $\psi$ is called as a K\"{a}hlerian twistor spinor. The defining equations above can be combined into a single equation for K\"{a}hlerian twistor spinors as in the following form
\begin{equation}
\nabla_X\psi=\frac{m+2}{8(r+1)(m-r+1)}(\widetilde{X}.\displaystyle{\not}D\psi+J\widetilde{X}.\displaystyle{\not}D^c\psi)+\frac{m-2r}{8(r+1)(m-r+1)}i(J\widetilde{X}.\displaystyle{\not}D\psi-\widetilde{X}.\displaystyle{\not}D^c\psi).
\end{equation}
\end{definition}
From the equations (23) and (24), some special types of K\"{a}hlerian twistor spinors can also be defined;
\begin{definition}
If a K\"{a}hlerian twistor spinor $\psi$ is in the kernel of the operator $\displaystyle{\not}D^+$, that is $\displaystyle{\not}D^+\psi=0$, then it is called as a holomorphic K\"{a}hlerian twistor spinor and is parallel with respect to all $X^-\in TM^-$. Similarly, if a K\"{a}hlerian twistor spinor $\psi$ is in the kernel of the operator $\displaystyle{\not}D^-$, that is $\displaystyle{\not}D^-\psi=0$, then it is called as an anti-holomorphic K\"{a}hlerian twistor spinor and is parallel with respect to all $X^+\in TM^+$.
\end{definition}

Special cases of Definition 2 are considered as the definition of K\"{a}hlerian twistor spinors in the literature. By considering any real numbers $a$ and $b$ and taking the first term on the right hand side of (25), one can write the equality
\begin{equation}
\nabla_X\psi=a\widetilde{X}.\displaystyle{\not}D\psi+bJ\widetilde{X}.\displaystyle{\not}D^c\psi
\end{equation}
which is the definition of K\"{a}hlerian twistor spinors given by Hijazi \cite{Hijazi2}. If we choose the real numbers $a=b=\frac{1}{4(r+1)}$ for $0\leq r\leq m-2$, we have
\begin{equation}
\nabla_X\psi=\frac{1}{4r}(\widetilde{X}.\displaystyle{\not}D\psi+J\widetilde{X}.\displaystyle{\not}D^c\psi)
\end{equation}
and this is the definition of K\"{a}hlerian twistor spinors given by Kirchberg \cite{Kirchberg2}. If we choose $a=\frac{1}{n}$ and $b=0$ in (26), then we obtain the Riemannian twistor equation given in (22).

For a K\"{a}hler manifold $M^{2m}$, if $m$ is even and the type of the K\"{a}hlerian twistor spinor is $r=\frac{m}{2}$, then the K\"{a}hlerian twistor equations (23) and (24) reduce to
\begin{equation}
\nabla_X\psi=\frac{1}{m+2}(\widetilde{X}^+.\displaystyle{\not}D^-\psi+\widetilde{X}^-.\displaystyle{\not}D^+\psi).
\end{equation}
If $M^{2m}$ is a compact K\"{a}hler spin manifold of positive constant scalar curvature, then (28) does not have a non-trivial solution and there are no K\"{a}hlerian twistor spinors of middle dimension type \cite{Pilca}. For a connected K\"{a}hler spin manifold $M^{2m}$, the dimension of the space of K\"{a}hlerian twistor spinors of type $r$, which is denoted by $KT(r)$, is bounded by \cite{Pilca}
\[
\dim(KT(r))\leq {m \choose r}+{m \choose r+1}+{m \choose r-1}.
\]

The classification of manifolds admitting K\"{a}hlerian twistor spinors are investigated in \cite{Pilca}. For a compact K\"{a}hlerian spin manifold $M^{2m}$ of positive constant scalar curvature, a K\"{a}hlerian twistor spinor $\psi\in\Sigma_rM$ of type $r$ is an anti-holomorphic K\"{a}hlerian twistor spinor if $r<\frac{m}{2}$ or a holomorphic K\"{a}hlerian twistor spinor if $r>\frac{m}{2}$. As a special case, weakly Bochner flat manifolds admit K\"{a}hlerian twistor spinors \cite{Pilca}.

\section{Bilinear forms}

For a spinor $\psi\in\Sigma M$, a dual spinor $\overline{\psi}\in\Sigma^*M$ can be defined by using the spin invariant inner product $(\, , \,)$ on $\Sigma M$. For $\psi, \phi\in\Sigma M$
\begin{eqnarray}
\Sigma^*M\times\Sigma M&\longrightarrow&\mathbb{F}\nonumber\\
\overline{\psi},\phi&\longmapsto&\overline{\psi}(\phi)=(\psi,\phi)\nonumber
\end{eqnarray}
where $\mathbb{F}$ is the division algebra on which the Clifford algebra is defined. The tensor product of spinors and dual spinors $\Sigma M\otimes\Sigma^*M$ acts on $\Sigma M$ by the Clifford multiplication. For $\psi,\phi,\kappa\in\Sigma M$ and $\overline{\phi}\in\Sigma^*M$, we have
\begin{equation}
(\psi\overline{\phi}).\kappa=(\phi,\kappa)\psi
\end{equation}
where we denote $\psi\otimes\overline{\phi}=\psi\overline{\phi}$. So, the elements of $\Sigma M\otimes\Sigma^*M$ can be regarded as linear transformations on $\Sigma M$ and hence they can be identified with the elements of the Clifford bundle $Cl(M)$. Then, for an orthonormal basis $\{e^a\}$, we can write $\psi\overline{\phi}$ as a sum of different degree differential forms as follows
\begin{eqnarray}
\psi\overline{\phi}&=&(\phi,\psi)+(\phi,e_a.\psi)e^a+(\phi,e_{ba}.\psi)e^{ab}+...\nonumber\\
&&+(\phi,e_{a_p...a_2a_1}.\psi)e^{a_1a_2...a_p}+...+(-1)^{\lfloor n/2\rfloor}(\phi, z.\psi)z
\end{eqnarray}
where $e^{a_1a_2...a_p}=e^{a_1}\wedge e^{a_2}\wedge...\wedge e^{a_p}$, $n$ is the dimension of the manifold and $z$ is the volume form. If we choose $\psi=\phi$, each term on the right hand side of (30) are called spinor bilinears of $\psi$ and the $p$-form bilinear is defined as
\begin{equation}
(\psi\overline{\psi})_p=(\psi,e_{a_p...a_2a_1}.\psi)e^{a_1a_2...a_p}.
\end{equation}
The inner product $(\, , \,)$ is said to have $\mathcal{J}$ involution with the following equality
\begin{equation}
(\phi, \omega.\psi)=(\omega^{\mathcal{J}}.\phi, \psi)
\end{equation}
for any inhomogeneous Clifford form $\omega$ and $\mathcal{J}$ can be $\xi$, $\xi^*$, $\xi\eta$ and $\xi\eta^*$ with $\xi$ and $\eta$ acts on a $p$-form $\alpha$ as $\alpha^{\xi}=(-1)^{\lfloor p/2\rfloor}\alpha$ and $\alpha^{\eta}=(-1)^p\alpha$ where $\lfloor\,\rfloor$ is the floor function taking the integral part of the argument and $^*$ is the complex conjugation. Moreover, for any Clifford forms $\alpha$ and $\beta$, we have $(\alpha.\beta)^{\xi}=\beta^{\xi}.\alpha^{\xi}$.

\begin{proposition}
For a Riemannian twistor spinor $\psi\in\Sigma M$ on a $n$-dimensional spin manifold $M$ which satisfies (22), all the $p$-form bilinears $(\psi\overline{\psi})_p$ satisfies the following conformal Killing-Yano (CKY) equation
\begin{equation}
\nabla_X(\psi\overline{\psi})_p=\frac{1}{p+1}i_Xd(\psi\overline{\psi})_p-\frac{1}{n-p+1}\widetilde{X}\wedge\delta(\psi\overline{\psi})_p
\end{equation}
for all vector fields $X\in TM$ and their metric duals $\widetilde{X}\in T^*M$.
\end{proposition}
The proof of the proposition can be found in \cite{Acik Ertem}. CKY forms correspond to antisymmetric generalizations of conformal Killing vector fields to higher degree differential forms.

On a K\"{a}hler manifold $M^{2m}$, the tensor product of spinors and dual spinors can also be written as a sum of different degree differential forms. Because of the decomposition of the exterior bundle given in Section II, the sum is written in terms of $(p,q)$-forms. For a spinor $\psi\in\Sigma M$ and its dual $\overline{\psi}\in\Sigma^*M$, we have the following decomposition
\begin{eqnarray}
\psi\overline{\psi}&=&(\psi, \psi)+(\psi, e_a.\psi)e^a+(\psi, Je_a.\psi)Je^a+(\psi, (e_b\wedge e_a).\psi)e^a\wedge e^b\nonumber\\
&&+(\psi, (e_b\wedge Je_a).\psi)Je^a\wedge e^b+(\psi, (Je_b\wedge Je_a).\psi)Je^a\wedge Je^b+ ... \nonumber\\
&&+(\psi, (e_{a_p}\wedge...\wedge e_{a_1}\wedge Je_{b_q}\wedge...\wedge Je_{b_1}).\psi)Je^{b_1}\wedge...\wedge Je^{b_q}\wedge e^{a_1}\wedge...\wedge e^{a_p}+...\\
&&+(\psi, (e_{a_m}\wedge...\wedge e_{a_1}\wedge Je_{b_m}\wedge...\wedge Je_{b_1}).\psi)Je^{b_1}\wedge...\wedge Je^{b_m}\wedge e^{a_1}\wedge...\wedge e^{a_m}.\nonumber
\end{eqnarray}
So, we can define the $(p,q)$-form bilinears of a spinor $\psi$ as
\begin{equation}
(\psi\overline{\psi})_{(p,q)}=(\psi, (e_{a_p}\wedge...\wedge e_{a_1}\wedge Je_{b_q}\wedge...\wedge Je_{b_1}).\psi)Je^{b_1}\wedge...\wedge Je^{b_q}\wedge e^{a_1}\wedge...\wedge e^{a_p}.
\end{equation}

\begin{theorem}
For a K\"{a}hlerian twistor spinor $\psi\in\Sigma M$ of type $r$ on a $2m$-dimensional K\"{a}hlerian spin manifold $M^{2m}$ which satisfies (25), the $(p,q)$-form bilinears $(\psi\overline{\psi})_{(p,q)}$ satisfies the following equation
\begin{eqnarray}
\nabla_{X_a}(\psi\overline{\psi})_{(p,q)}&=&\frac{1}{p+1}i_{X_a}\bigg(d(\psi\overline{\psi})_{(p,q)}-2L\alpha_{(p,q-1)}-2J\beta_{(p,q+1)}\bigg)\nonumber\\
&&-\frac{1}{m-p+1}e_a\wedge\bigg(\delta(\psi\overline{\psi})_{(p,q)}-J\alpha_{(p,q-1)}-2\Lambda\beta_{(p,q+1)}\bigg)\nonumber\\
&&+\frac{1}{q+1}i_{JX_a}\bigg(d^c(\psi\overline{\psi})_{(p,q)}+2L\gamma_{(p-1,q)}-2J\mu_{(p+1,q)}\bigg)\\
&&-\frac{1}{m-q+1}Je_a\wedge\bigg(\delta^c(\psi\overline{\psi})_{(p,q)}-J\gamma_{(p-1,q)}-2\Lambda\mu_{(p+1,q)}\bigg).\nonumber
\end{eqnarray}
Here, the forms $\alpha_{(p,q-1)}$, $\beta_{(p,q+1)}$, $\gamma_{(p-1,q)}$ and $\mu_{(p+1,q)}$ are defined as follows
\begin{eqnarray}
\alpha_{(p,q-1)}&=&\big(k\mathcal{A}+l\mathcal{B}\big)_{(p,q-1)}\\
\beta_{(p,q+1)}&=&\big(k\mathcal{C}+l\mathcal{D}\big)_{(p,q+1)}\\
\gamma_{(p-1,q)}&=&\big(k\mathcal{B}-l\mathcal{A}\big)_{(p-1,q)}\\
\mu_{(p+1,q)}&=&\big(k\mathcal{D}-l\mathcal{C}\big)_{(p+1,q)}
\end{eqnarray}
where
\begin{eqnarray}
\mathcal{A}&=&\displaystyle{\not}d^c(\psi\overline{\psi})-2Je^b\wedge(\psi\overline{\nabla_{X_b}\psi})\\
\mathcal{B}&=&\displaystyle{\not}d(\psi\overline{\psi})-2e^b\wedge(\psi\overline{\nabla_{X_b}\psi})\\
\mathcal{C}&=&\displaystyle{\not}d^c(\psi\overline{\psi})-2i_{JX^b}(\psi\overline{\nabla_{X_b}\psi})\\
\mathcal{D}&=&\displaystyle{\not}d(\psi\overline{\psi})-2i_{X^b}(\psi\overline{\nabla_{X_b}\psi})
\end{eqnarray}
and
\begin{equation}
k:=\frac{m+2}{8(r+1)(m-r+1)}\quad\quad , \quad\quad l:=\frac{m-2r}{8(r+1)(m-r+1)}.
\end{equation}
\end{theorem}
\begin{proof}
First we calculate the covariant derivative of the $(p,q)$-form bilinears $(\psi\overline{\psi})_{(p,q)}$ of a K\"{a}hlerian twistor spinor $\psi$ satisfying (25). Since the covariant derivative $\nabla$ is compatible with the spin invariant inner product $(\, , \,)$, we can write
\begin{eqnarray}
\nabla_{X_a}(\psi\overline{\psi})_{(p,q)}=\big((\nabla_{X_a}\psi)\overline{\psi}\big)_{(p,q)}+\big(\psi\overline{\nabla_{X_a}\psi}\big)_{(p,q)}\nonumber
\end{eqnarray}
and from (25), we have
\begin{eqnarray}
\nabla_{X_a}(\psi\overline{\psi})_{(p,q)}&=&k\bigg[\big((e_a.\displaystyle{\not}D\psi)\overline{\psi}\big)_{(p,q)}+\big((Je_a.\displaystyle{\not}D^c\psi)\overline{\psi}\big)_{(p,q)}+\big(\psi\overline{e_a.\displaystyle{\not}D\psi}\big)_{(p,q)}+\big(\psi\overline{Je_a.\displaystyle{\not}D^c\psi}\big)_{(p,q)}\bigg]\nonumber\\
&&+l\bigg[\big((Je_a.\displaystyle{\not}D\psi)\overline{\psi}\big)_{(p,q)}-\big((e_a.\displaystyle{\not}D^c\psi)\overline{\psi}\big)_{(p,q)}+\big(\psi\overline{Je_a.\displaystyle{\not}D\psi}\big)_{(p,q)}-\big(\psi\overline{e_a.\displaystyle{\not}D^c\psi}\big)_{(p,q)}\bigg]\nonumber
\end{eqnarray}
where we have used the definitions of $k$ and $l$ given in (45). From the equalities (29) and (32), we can write the identites
\begin{eqnarray}
\overline{e_a.e^b.\nabla_{X_b}\psi}&=&\overline{\nabla_{X_b}\psi}.{e^b}^{\mathcal{J}}.{e_a}^{\mathcal{J}}\nonumber\\
&=&\overline{\nabla_{X_b}\psi}.e^b.e_a\nonumber
\end{eqnarray}
which is true for all choices of $\mathcal{J}$. Similar identites can also be obtained for the terms including $Je_a$ basis. By using these identites and the definitions of $\displaystyle{\not}D$ and $\displaystyle{\not}D^c$ given in (19) and (20), we obtain
\begin{eqnarray}
\nabla_{X_a}(\psi\overline{\psi})_{(p,q)}&=&k\bigg[\big(e_a.e^b.(\nabla_{X_b}\psi)\overline{\psi}\big)_{(p,q)}+\big(Je_a.Je^b.(\nabla_{X_b}\psi)\overline{\psi}\big)_{(p,q)}+\big(\psi\overline{\nabla_{X_b}\psi}.e^b.e_a\big)_{(p,q)}+\big(\psi\overline{\nabla_{X_b}\psi}.Je^b.Je_a\big)_{(p,q)}\bigg]\nonumber\\
&&+l\bigg[\big(Je_a.e^b.(\nabla_{X_b}\psi)\overline{\psi}\big)_{(p,q)}-\big(e_a.Je^b.(\nabla_{X_b}\psi)\overline{\psi}\big)_{(p,q)}+\big(\psi\overline{\nabla_{X_b}\psi}.e^b.Je_a\big)_{(p,q)}-\big(\psi\overline{\nabla_{X_b}\psi}.Je^b.e_a\big)_{(p,q)}\bigg]\nonumber
\end{eqnarray}
and from the identity $(\nabla_{X_b}\psi)\overline{\psi}=\nabla_{X_b}(\psi\overline{\psi})-\psi\overline{\nabla_{X_b}\psi}$, we can write
\begin{eqnarray}
\nabla_{X_a}(\psi\overline{\psi})_{(p,q)}&=&k\bigg[\big(e_a.e^b.\nabla_{X_b}(\psi\overline{\psi})\big)_{(p,q)}-\big(e_a.e^b.\psi\overline{\nabla_{X_b}\psi}\big)_{(p,q)}+\big(Je_a.Je^b.\nabla_{X_b}(\psi\overline{\psi})\big)_{(p,q)}\nonumber\\
&&-\big(Je_a.Je^b.\psi\overline{\nabla_{X_b}\psi}\big)_{(p,q)}+\big(\psi\overline{\nabla_{X_b}\psi}.e^b.e_a\big)_{(p,q)}+\big(\psi\overline{\nabla_{X_b}\psi}.Je^b.Je_a\big)_{(p,q)}\bigg]\nonumber\\
&&+l\bigg[\big(Je_a.e^b.\nabla_{X_b}(\psi\overline{\psi})\big)_{(p,q)}-\big(Je_a.e^b.\psi\overline{\nabla_{X_b}\psi}\big)_{(p,q)}-\big(e_a.Je^b.\nabla_{X_b}(\psi\overline{\psi})\big)_{(p,q)}\nonumber\\
&&+\big(e_a.Je^b.\psi\overline{\nabla_{X_b}\psi}\big)_{(p,q)}+\big(\psi\overline{\nabla_{X_b}\psi}.e^b.Je_a\big)_{(p,q)}-\big(\psi\overline{\nabla_{X_b}\psi}.Je^b.e_a\big)_{(p,q)}\bigg].\nonumber
\end{eqnarray}
For any Clifford form $\omega$, one can write the Clifford product of a 1-form $e^a$ and $\omega$ in terms of wedge product and interior product as follows
\begin{eqnarray}
e^a.\omega&=&e^a\wedge\omega+i_{X^a}\omega\nonumber\\
\omega.e^a&=&e^a\wedge\omega^{\eta}-i_{X^a}\omega^{\eta}.
\end{eqnarray}
From this property, we can write the identites
\begin{eqnarray}
\big(\psi\overline{\nabla_{X_b}\psi}.e^b.e_a\big)_{(p,q)}-\big(e_a.e^b.\psi\overline{\nabla_{X_b}\psi}\big)_{(p,q)}&=&-2e_a\wedge e^b\wedge\big(\psi\overline{\nabla_{X_b}\psi}\big)_{(p-2,q)}-2i_{X_a}i_{X^b}\big(\psi\overline{\nabla_{X_b}\psi}\big)_{(p+2,q)}\nonumber\\
\big(\psi\overline{\nabla_{X_b}\psi}.Je^b.Je_a\big)_{(p,q)}-\big(Je_a.Je^b.\psi\overline{\nabla_{X_b}\psi}\big)_{(p,q)}&=&-2Je_a\wedge Je^b\wedge\big(\psi\overline{\nabla_{X_b}\psi}\big)_{(p,q-2)}-2i_{JX_a}i_{JX^b}\big(\psi\overline{\nabla_{X_b}\psi}\big)_{(p,q+2)}\nonumber\\
\big(\psi\overline{\nabla_{X_b}\psi}.e^b.Je_a\big)_{(p,q)}-\big(Je_a.e^b.\psi\overline{\nabla_{X_b}\psi}\big)_{(p,q)}&=&-2Je_a\wedge e^b\wedge\big(\psi\overline{\nabla_{X_b}\psi}\big)_{(p-1,q-1)}-2i_{JX_a}i_{X^b}\big(\psi\overline{\nabla_{X_b}\psi}\big)_{(p+1,q+1)}\nonumber\\
-\big(\psi\overline{\nabla_{X_b}\psi}.Je^b.e_a\big)_{(p,q)}-\big(e_a.Je^b.\psi\overline{\nabla_{X_b}\psi}\big)_{(p,q)}&=&2e_a\wedge Je^b\wedge\big(\psi\overline{\nabla_{X_b}\psi}\big)_{(p-1,q-1)}+2i_{X_a}i_{JX^b}\big(\psi\overline{\nabla_{X_b}\psi}\big)_{(p+1,q+1)}.\nonumber
\end{eqnarray}
So, from these identities and the definitions of $\displaystyle{\not}d=e^a.\nabla_{X_a}$ and $\displaystyle{\not}d^c=Je^a.\nabla_{X_a}$ acting on Clifford forms, the covariant derivative of $(\psi\overline{\psi})_{(p,q)}$ is written as
\begin{eqnarray}
\nabla_{X_a}(\psi\overline{\psi})_{(p,q)}&=&k\bigg[\big(e_a.\displaystyle{\not}d(\psi\overline{\psi})\big)_{(p,q)}-2e_a\wedge e^b\wedge\big(\psi\overline{\nabla_{X_b}\psi}\big)_{(p-2,q)}-2i_{X_a}i_{X^b}\big(\psi\overline{\nabla_{X_b}\psi}\big)_{(p+2,q)}\nonumber\\
&&+\big(Je_a.\displaystyle{\not}d^c(\psi\overline{\psi})\big)_{(p,q)}-2Je_a\wedge Je^b\wedge\big(\psi\overline{\nabla_{X_b}\psi}\big)_{(p,q-2)}-2i_{JX_a}i_{JX^b}\big(\psi\overline{\nabla_{X_b}\psi}\big)_{(p,q+2)}\bigg]\nonumber\\
&&+l\bigg[\big(Je_a.\displaystyle{\not}d(\psi\overline{\psi})\big)_{(p,q)}-2Je_a\wedge e^b\wedge\big(\psi\overline{\nabla_{X_b}\psi}\big)_{(p-1,q-1)}-2i_{JX_a}i_{X^b}\big(\psi\overline{\nabla_{X_b}\psi}\big)_{(p+1,q+1)}\nonumber\\
&&-\big(e_a.\displaystyle{\not}d^c(\psi\overline{\psi})\big)_{(p,q)}+2e_a\wedge Je^b\wedge\big(\psi\overline{\nabla_{X_b}\psi}\big)_{(p-1,q-1)}+2i_{X_a}i_{JX^b}\big(\psi\overline{\nabla_{X_b}\psi}\big)_{(p+1,q+1)}\bigg]\nonumber
\end{eqnarray}
and by using (46) again, we finally obtain
\begin{eqnarray}
\nabla_{X_a}(\psi\overline{\psi})_{(p,q)}&=&k\bigg[e_a\wedge\big(\displaystyle{\not}d(\psi\overline{\psi})-2e^b\wedge(\psi\overline{\nabla_{X_b}\psi})\big)_{(p-1,q)}+i_{X_a}\big(\displaystyle{\not}d(\psi\overline{\psi})-2i_{X^b}(\psi\overline{\nabla_{X_b}\psi})\big)_{(p+1,q)}\nonumber\\
&&+Je_a\wedge\big(\displaystyle{\not}d^c(\psi\overline{\psi})-2Je^b\wedge(\psi\overline{\nabla_{X_b}\psi})\big)_{(p,q-1)}+i_{JX_a}\big(\displaystyle{\not}d^c(\psi\overline{\psi})-2i_{JX^b}(\psi\overline{\nabla_{X_b}\psi})\big)_{(p,q+1)}\bigg]\nonumber\\
&&+l\bigg[Je_a\wedge\big(\displaystyle{\not}d(\psi\overline{\psi})-2e^b\wedge(\psi\overline{\nabla_{X_b}\psi})\big)_{(p,q-1)}+i_{JX_a}\big(\displaystyle{\not}d(\psi\overline{\psi})-2i_{X^b}(\psi\overline{\nabla_{X_b}\psi})\big)_{(p,q+1)}\nonumber\\
&&-e_a\wedge\big(\displaystyle{\not}d^c(\psi\overline{\psi})-2Je^b\wedge(\psi\overline{\nabla_{X_b}\psi})\big)_{(p-1,q)}-i_{X_a}\big(\displaystyle{\not}d^c(\psi\overline{\psi})-2i_{JX^b}(\psi\overline{\nabla_{X_b}\psi})\big)_{(p+1,q)}\bigg].
\end{eqnarray}
Now, by using the equalities (3), (4), (12) and (13), we can calculate the action of the operators $d$, $\delta$, $d^c$ and $\delta^c$ on the bilinears $(\psi\overline{\psi})_{(p,q)}$. Because of the antisymmetry of the wedge product, we have $e^a\wedge e_a=0$ and from the property $e^a\wedge i_{X_a}\alpha_{(p,q)}=p\alpha_{(p,q)}$ for any $(p,q)$-form $\alpha$, we obtain
\begin{eqnarray}
d(\psi\overline{\psi})_{(p,q)}&=&e^a\wedge\nabla_{X_a}(\psi\overline{\psi})_{(p,q)}\nonumber\\
&=&k\bigg[(p+1)\big(\displaystyle{\not}d(\psi\overline{\psi})-2i_{X^b}(\psi\overline{\nabla_{X_b}\psi})\big)_{(p+1,q)}+e^ a\wedge Je_a\wedge\big(\displaystyle{\not}d^c(\psi\overline{\psi})-2Je^b\wedge(\psi\overline{\nabla_{X_b}\psi})\big)_{(p,q-1)}\nonumber\\
&&+e^a\wedge i_{JX_a}\big(\displaystyle{\not}d^c(\psi\overline{\psi})-2i_{JX^b}(\psi\overline{\nabla_{X_b}\psi})\big)_{(p,q+1)}\bigg]\nonumber\\
&&+l\bigg[e^a\wedge Je_a\wedge\big(\displaystyle{\not}d(\psi\overline{\psi})-2e^b\wedge(\psi\overline{\nabla_{X_b}\psi})\big)_{(p,q-1)}+e^a\wedge i_{JX_a}\big(\displaystyle{\not}d(\psi\overline{\psi})-2i_{X^b}(\psi\overline{\nabla_{X_b}\psi})\big)_{(p,q+1)}\nonumber\\
&&-(p+1)\big(\displaystyle{\not}d^c(\psi\overline{\psi})-2i_{JX^b}(\psi\overline{\nabla_{X_b}\psi})\big)_{(p+1,q)}\nonumber\\
&=&(p+1)\bigg(k\big(\displaystyle{\not}d(\psi\overline{\psi})-2i_{X^b}(\psi\overline{\nabla_{X_b}\psi})\big)-l\big(\displaystyle{\not}d^c(\psi\overline{\psi}-2i_{JX^b}(\psi\overline{\nabla_{X_b}\psi}))\big)\bigg)_{(p+1,q)}\nonumber\\
&&+e^a\wedge Je_a\wedge\bigg(k\big(\displaystyle{\not}d^c(\psi\overline{\psi})-2Je^b\wedge(\psi\overline{\nabla_{X_b}\psi})\big)+l\big(\displaystyle{\not}d(\psi\overline{\psi})-2e^b\wedge(\psi\overline{\nabla_{X_b}\psi})\big)\bigg)_{(p,q-1)}\nonumber\\
&&+e^a\wedge i_{JX_a}\bigg(k\big(\displaystyle{\not}d^c(\psi\overline{\psi})-2i_{JX^b}(\psi\overline{\nabla_{X_b}\psi})\big)+l\big(\displaystyle{\not}d(\psi\overline{\psi})-2i_{X^b}(\psi\overline{\nabla_{X_b}\psi})\big)\bigg)_{(p,q+1)}.
\end{eqnarray}
In a similar way, by considering the equalities $Je^a\wedge Je_a=0$ and $Je^a\wedge i_{JX_a}\alpha_{(p,q)}=q\alpha_{(p,q)}$, we can find
\begin{eqnarray}
d^c(\psi\overline{\psi})_{(p,q)}&=&Je^a\wedge\nabla_{X_a}(\psi\overline{\psi})_{(p,q)}\nonumber\\
&=&(q+1)\bigg(k\big(\displaystyle{\not}d^c(\psi\overline{\psi})-2i_{JX^b}(\psi\overline{\nabla_{X_b}\psi})\big)+l\big(\displaystyle{\not}d(\psi\overline{\psi}-2i_{X^b}(\psi\overline{\nabla_{X_b}\psi}))\big)\bigg)_{(p,q+1)}\nonumber\\
&&+Je^a\wedge e_a\wedge\bigg(k\big(\displaystyle{\not}d(\psi\overline{\psi})-2e^b\wedge(\psi\overline{\nabla_{X_b}\psi})\big)-l\big(\displaystyle{\not}d^c(\psi\overline{\psi})-2Je^b\wedge(\psi\overline{\nabla_{X_b}\psi})\big)\bigg)_{(p-1,q)}\nonumber\\
&&+Je^a\wedge i_{X_a}\bigg(k\big(\displaystyle{\not}d(\psi\overline{\psi})-2i_{X^b}(\psi\overline{\nabla_{X_b}\psi})\big)-l\big(\displaystyle{\not}d^c(\psi\overline{\psi})-2i_{JX^b}(\psi\overline{\nabla_{X_b}\psi})\big)\bigg)_{(p+1,q)}.
\end{eqnarray}
We also have the identities $i_{X^a}i_{X_a}=0$, $i_{X^a}(Je_a)=0$ and $i_{X^a}(e_a\wedge\alpha_{(p,q)})=(m-p)\alpha_{(p,q)}$. By using them, we find
\begin{eqnarray}
\delta(\psi\overline{\psi})_{(p,q)}&=&-i_{X^a}\nabla_{X_a}(\psi\overline{\psi})_{(p,q)}\nonumber\\
&=&(m-p+1)\bigg(-k\big(\displaystyle{\not}d(\psi\overline{\psi})-2e^b\wedge(\psi\overline{\nabla_{X_b}\psi})\big)+l\big(\displaystyle{\not}d^c(\psi\overline{\psi}-2Je^b\wedge(\psi\overline{\nabla_{X_b}\psi}))\big)\bigg)_{(p-1,q)}\nonumber\\
&&+Je^a\wedge i_{X_a}\bigg(k\big(\displaystyle{\not}d^c(\psi\overline{\psi})-2Je^b\wedge(\psi\overline{\nabla_{X_b}\psi})\big)+l\big(\displaystyle{\not}d(\psi\overline{\psi})-2e^b\wedge(\psi\overline{\nabla_{X_b}\psi})\big)\bigg)_{(p,q-1)}\nonumber\\
&&+i_{X^a} i_{JX_a}\bigg(k\big(\displaystyle{\not}d^c(\psi\overline{\psi})-2i_{JX^b}(\psi\overline{\nabla_{X_b}\psi})\big)+l\big(\displaystyle{\not}d(\psi\overline{\psi})-2i_{X^b}(\psi\overline{\nabla_{X_b}\psi})\big)\bigg)_{(p,q+1)}.
\end{eqnarray}
Similarly, from the identities $i_{JX^a}i_{JX_a}=0$, $i_{JX^a}(e_a)=0$ and $i_{JX^a}(Je_a\wedge\alpha_{(p,q)})=(m-q)\alpha_{(p,q)}$, we obtain
\begin{eqnarray}
\delta^c(\psi\overline{\psi})_{(p,q)}&=&-i_{JX^a}\nabla_{X_a}(\psi\overline{\psi})_{(p,q)}\nonumber\\
&=&-(m-q+1)\bigg(k\big(\displaystyle{\not}d^c(\psi\overline{\psi})-2Je^b\wedge(\psi\overline{\nabla_{X_b}\psi})\big)+l\big(\displaystyle{\not}d(\psi\overline{\psi}-2e^b\wedge(\psi\overline{\nabla_{X_b}\psi}))\big)\bigg)_{(p,q-1)}\nonumber\\
&&+e^a\wedge i_{JX_a}\bigg(k\big(\displaystyle{\not}d(\psi\overline{\psi})-2e^b\wedge(\psi\overline{\nabla_{X_b}\psi})\big)-l\big(\displaystyle{\not}d^c(\psi\overline{\psi})-2Je^b\wedge(\psi\overline{\nabla_{X_b}\psi})\big)\bigg)_{(p-1,q)}\nonumber\\
&&+i_{JX^a} i_{X_a}\bigg(-k\big(\displaystyle{\not}d(\psi\overline{\psi})-2i_{X^b}(\psi\overline{\nabla_{X_b}\psi})\big)+l\big(\displaystyle{\not}d^c(\psi\overline{\psi})-2i_{JX^b}(\psi\overline{\nabla_{X_b}\psi})\big)\bigg)_{(p+1,q)}.
\end{eqnarray}
The next step in the proof consists of the contractions of $d(\psi\overline{\psi})_{(p,q)}$ and $d^c(\psi\overline{\psi})_{(p,q)}$ given in (48) and (49) with the vector fields $X_a$ and $JX_a$ respectively and the wedge products of $\delta(\psi\overline{\psi})_{(p,q)}$ and $\delta^c(\psi\overline{\psi})_{(p,q)}$ given in (50) and (51) with the 1-forms $e_a$ and $Je_a$ respectively. If we contract $d(\psi\overline{\psi})_{(p,q)}$ with $X_a$, we find from (48)
\begin{eqnarray}
i_{X_a}d(\psi\overline{\psi})_{(p,q)}&=&(p+1)i_{X_a}\bigg(k\big(\displaystyle{\not}d(\psi\overline{\psi})-2i_{X^b}(\psi\overline{\nabla_{X_b}\psi})\big)-l\big(\displaystyle{\not}d^c(\psi\overline{\psi})-2i_{JX^b}(\psi\overline{\nabla_{X_b}\psi})\big)\bigg)_{(p+1,q)}\nonumber\\
&&+2i_{X_a}\bigg[L\bigg(k\big(\displaystyle{\not}d^c(\psi\overline{\psi})-2Je^b\wedge(\psi\overline{\nabla_{X_b}\psi})\big)+l\big(\displaystyle{\not}d(\psi\overline{\psi})-2e^b\wedge(\psi\overline{\nabla_{X_b}\psi})\big)\bigg)_{(p,q-1)}\bigg]\nonumber\\
&&+2i_{JX_a}\bigg(k\big(\displaystyle{\not}d^c(\psi\overline{\psi})-2i_{JX^b}(\psi\overline{\nabla_{X_b}\psi})\big)+l\big(\displaystyle{\not}d(\psi\overline{\psi})-2i_{X^b}(\psi\overline{\nabla_{X_b}\psi})\big)\bigg)_{(p,q+1)}
\end{eqnarray}
where we have used the definitions of $L$ and $J$ given in (5) and (7), the identity $i_{JX_a}=Ji_{X_a}$ and the commutator relations in (11). Similarly, by contracting $d^c(\psi\overline{\psi})_{(p,q)}$ with $JX_a$, we obtain from (49)
\begin{eqnarray}
i_{JX_a}d^c(\psi\overline{\psi})_{(p,q)}&=&(q+1)i_{JX_a}\bigg(k\big(\displaystyle{\not}d^c(\psi\overline{\psi})-2i_{JX^b}(\psi\overline{\nabla_{X_b}\psi})\big)+l\big(\displaystyle{\not}d(\psi\overline{\psi})-2i_{X^b}(\psi\overline{\nabla_{X_b}\psi})\big)\bigg)_{(p,q+1)}\nonumber\\
&&-2i_{JX_a}\bigg[L\bigg(k\big(\displaystyle{\not}d(\psi\overline{\psi})-2e^b\wedge(\psi\overline{\nabla_{X_b}\psi})\big)-l\big(\displaystyle{\not}d^c(\psi\overline{\psi})-2Je^b\wedge(\psi\overline{\nabla_{X_b}\psi})\big)\bigg)_{(p-1,q)}\bigg]\nonumber\\
&&-2i_{X_a}\bigg(k\big(\displaystyle{\not}d(\psi\overline{\psi})-2i_{X^b}(\psi\overline{\nabla_{X_b}\psi})\big)-l\big(\displaystyle{\not}d^c(\psi\overline{\psi})-2i_{JX^b}(\psi\overline{\nabla_{X_b}\psi})\big)\bigg)_{(p+1,q)}.
\end{eqnarray}
By taking the wedge product of $e_a$ with $\delta(\psi\overline{\psi})_{(p,q)}$, we find from (50)
\begin{eqnarray}
e_a\wedge\delta(\psi\overline{\psi})_{(p,q)}&=&(m-p+1)\bigg(-k\big(\displaystyle{\not}d(\psi\overline{\psi})-2e^b\wedge(\psi\overline{\nabla_{X_b}\psi})\big)+l\big(\displaystyle{\not}d^c(\psi\overline{\psi})-2Je^b\wedge(\psi\overline{\nabla_{X_b}\psi})\big)\bigg)_{(p-1,q)}\nonumber\\
&&+e_a\wedge J\bigg(k\big(\displaystyle{\not}d^c(\psi\overline{\psi})-2Je^b\wedge(\psi\overline{\nabla_{X_b}\psi})\big)+l\big(\displaystyle{\not}d(\psi\overline{\psi})-2e^b\wedge(\psi\overline{\nabla_{X_b}\psi})\big)\bigg)_{(p,q-1)}\nonumber\\
&&+2e_a\wedge\Lambda\bigg(k\big(\displaystyle{\not}d^c(\psi\overline{\psi})-2i_{JX^b}(\psi\overline{\nabla_{X_b}\psi})\big)+l\big(\displaystyle{\not}d(\psi\overline{\psi})-2i_{X^b}(\psi\overline{\nabla_{X_b}\psi})\big)\bigg)_{(p,q+1)}
\end{eqnarray}
where we have used the definitons of $J$ and $\Lambda$ given in (7) and (6). Similarly, we can also find the wedge product of $Je_a$ with $\delta^c(\psi\overline{\psi})_{(p,q)}$ from (51) as
\begin{eqnarray}
Je_a\wedge\delta^c(\psi\overline{\psi})_{(p,q)}&=&-(m-q+1)Je_a\wedge\bigg(k\big(\displaystyle{\not}d^c(\psi\overline{\psi})-2Je^b\wedge(\psi\overline{\nabla_{X_b}\psi})\big)+l\big(\displaystyle{\not}d(\psi\overline{\psi})-2e^b\wedge(\psi\overline{\nabla_{X_b}\psi})\big)\bigg)_{(p,q-1)}\nonumber\\
&&+Je_a\wedge J\bigg(k\big(\displaystyle{\not}d(\psi\overline{\psi})-2e^b\wedge(\psi\overline{\nabla_{X_b}\psi})\big)-l\big(\displaystyle{\not}d^c(\psi\overline{\psi})-2Je^b\wedge(\psi\overline{\nabla_{X_b}\psi})\big)\bigg)_{(p-1,q)}\nonumber\\
&&+2Je_a\wedge\Lambda\bigg(-k\big(\displaystyle{\not}d(\psi\overline{\psi})-2i_{X^b}(\psi\overline{\nabla_{X_b}\psi})\big)+l\big(\displaystyle{\not}d^c(\psi\overline{\psi})-2i_{JX^b}(\psi\overline{\nabla_{X_b}\psi})\big)\bigg)_{(p+1,q)}.
\end{eqnarray}
Hence, by considering the definitions $\alpha_{(p,q-1)}$, $\beta_{(p,q+1)}$, $\gamma_{(p-1,q)}$ and $\mu_{(p+1,q)}$ given in (37)-(40) and (41)-(44) and comparing the equality obtained in (47) with the equalities (52)-(55), we finally find the following equation
\begin{eqnarray}
\nabla_{X_a}(\psi\overline{\psi})_{(p,q)}&=&\frac{1}{p+1}i_{X_a}\bigg(d(\psi\overline{\psi})_{(p,q)}-2L\alpha_{(p,q-1)}-2J\beta_{(p,q+1)}\bigg)\nonumber\\
&&-\frac{1}{m-p+1}e_a\wedge\bigg(\delta(\psi\overline{\psi})_{(p,q)}-J\alpha_{(p,q-1)}-2\Lambda\beta_{(p,q+1)}\bigg)\nonumber\\
&&+\frac{1}{q+1}i_{JX_a}\bigg(d^c(\psi\overline{\psi})_{(p,q)}+2L\gamma_{(p-1,q)}-2J\mu_{(p+1,q)}\bigg)\nonumber\\
&&-\frac{1}{m-q+1}Je_a\wedge\bigg(\delta^c(\psi\overline{\psi})_{(p,q)}-J\gamma_{(p-1,q)}-2\Lambda\mu_{(p+1,q)}\bigg)
\end{eqnarray}
which proves the theorem.
\end{proof}

\section{Special Cases}

Bilinear form equation (56) for K\"{a}hlerian twistor spinors reduces to a more simple form for special cases of Kirchberg's and Hijazi's K\"{a}hlerian twistor spinors given in (27) and (26), respectively.

\begin{corollary}
Bilinears forms of Kirchberg's K\"{a}hlerian twistor spinors given in (27) satisfy the equation (56) with $k=\frac{1}{4r}$ and $l=0$ in the definitions (37)-(40). Similarly, bilinear forms of Hijazi's K\"{a}hlerian twistor spinors given in (26) satisfy the equation (56) with the following new definitions
\begin{eqnarray}
\alpha_{p,q-1}&=&b\bigg(\displaystyle{\not}d^c(\psi\overline{\psi})-2Je^b\wedge(\psi\overline{\nabla_{X_b}\psi})\bigg)_{(p,q-1)}\\
\beta_{(p,q+1)}&=&b\bigg(\displaystyle{\not}d^c(\psi\overline{\psi})-2i_{JX^b}(\psi\overline{\nabla_{X_b}\psi})\bigg)_{(p,q+1)}\\
\gamma_{(p-1,q)}&=&a\bigg(\displaystyle{\not}d(\psi\overline{\psi})-2e^b\wedge(\psi\overline{\nabla_{X_b}\psi})\bigg)_{(p-1,q)}\\
\mu_{(p+1,q)}&=&a\bigg(\displaystyle{\not}d(\psi\overline{\psi})-2i_{X^b}(\psi\overline{\nabla_{X_b}\psi})\bigg)_{(p,q-1)}
\end{eqnarray}
where $a$ and $b$ are constants given in (26).
\end{corollary}

We can also determine the bilinear form equations of holomorphic and anti-holomorphic K\"{a}hlerian twistor spinors of type $r$.
\begin{theorem}
If $\psi$ is a holomorphic K\"{a}hlerian twistor spinor of type $r$ on a K\"{a}hlerian spin manifold $M^{2m}$, then the bilinear forms $(\psi\overline{\psi})_p$ satisfy the equation (56) with the new definitions of
\begin{equation}
k=-l=\frac{1}{16(m-r+1)}.
\end{equation}
Similarly, if $\psi$ is an anti-holomorphic K\"{a}hlerian twistor spinor of type $r$ on a K\"{a}hlerian spin manifold $M^{2m}$, then the bilinear forms $(\psi\overline{\psi})_p$ satisfy the equation (56) with the new definitions of
\begin{equation}
k=l=\frac{1}{16(r+1)}.
\end{equation}
\end{theorem}
\begin{proof}
If $\psi$ is a K\"{a}hlerian twistor spinor of type $r$, then it satisfies
\begin{equation}
\nabla_X\psi=\frac{m+2}{8(r+1)(m-r+1)}\big(\widetilde{X}.\displaystyle{\not}D\psi+J\widetilde{X}.\displaystyle{\not}D^c\psi\big)+\frac{m-2r}{8(r+1)(m-r+1)}i\big(J\widetilde{X}.\displaystyle{\not}D\psi-\widetilde{X}.\displaystyle{\not}D^c\psi\big).
\end{equation}
By considering the definitions
\[
X=X^++X^-\quad\quad, \quad\quad JX=i(X^+-X^-)
\]
and
\[
\displaystyle{\not}D=\displaystyle{\not}D^++\displaystyle{\not}D^-\quad\quad ,\quad\quad \displaystyle{\not}D^c=i(\displaystyle{\not}D^+-\displaystyle{\not}D^-),
\]
one can write the above equation in the following form
\begin{equation}
\nabla_X\psi=\frac{1}{4(m-r+1)}\widetilde{X}^+.\displaystyle{\not}D^+\psi-\frac{1}{4(r+1)}\widetilde{X}^-.\displaystyle{\not}D^+\psi.
\end{equation}
If $\psi$ is holomorphic, then we have $\displaystyle{\not}D^+\psi=0$ and (64) transforms into
\begin{eqnarray}
\nabla_X\psi&=&\frac{1}{4(m-r+1)}\widetilde{X}^+.\displaystyle{\not}D^-\psi\nonumber\\
&=&\frac{1}{16(m-r+1)}\big(\widetilde{X}.\displaystyle{\not}D\psi+i\widetilde{X}.\displaystyle{\not}D^c\psi-iJ\widetilde{X}.\displaystyle{\not}D\psi+J\widetilde{X}.\displaystyle{\not}D^c\psi\big)
\end{eqnarray}
where we have used $X^+=\frac{1}{2}(X-iJX)$ and $\displaystyle{\not}D^-=\frac{1}{2}(\displaystyle{\not}D+i\displaystyle{\not}D^c)$. So, (65) is exactly in the same form with K\"{a}hlerian twistor equation for $k=-l=\frac{1}{16(m-r+1)}$ and the bilinears $(\psi\overline{\psi})_p$ satisfy exactly the same equation for these new values of $k$ and $l$.

Similarly, if $\psi$ is anti-holomorphic, then we have $\displaystyle{\not}D^-\psi=0$ and (64) transforms into
\begin{eqnarray}
\nabla_X\psi&=&\frac{1}{4(r+1)}\widetilde{X}^-.\displaystyle{\not}D^+\psi\nonumber\\
&=&\frac{1}{16(r+1)}\big(\widetilde{X}.\displaystyle{\not}D\psi-i\widetilde{X}.\displaystyle{\not}D^c\psi+iJ\widetilde{X}.\displaystyle{\not}D\psi+J\widetilde{X}.\displaystyle{\not}D^c\psi\big).
\end{eqnarray}
So, (66) is also exactly in the same form with K\"{a}hlerian twistor equation for $k=l=\frac{1}{16(r+1)}$ and the bilinears $(\psi\overline{\psi})_p$ satisfy exactly the same equation for these new values of $k$ and $l$.
\end{proof}

We obtain another special case for the bilinear forms of K\"{a}hlerian twistor spinors when some constraints on $\alpha_{(p,q-1)}$, $\beta_{(p,q+1)}$, $\gamma_{(p-1,q)}$ and $\mu_{(p+1,q)}$ given in (37)-(40) are satisfied.
\begin{proposition}
For a K\"{a}hlerian twistor spinor $\psi$ of type $r$, if the forms $\alpha_{(p,q-1)}$, $\beta_{(p,q+1)}$, $\gamma_{(p-1,q)}$ and $\mu_{(p+1,q)}$ defined in (37)-(40) satisfy the following conditions
\begin{eqnarray}
L\alpha_{(p,q-1)}=-J\beta_{(p,q+1)}\quad\quad &,&\quad\quad J\alpha_{(p,q-1)}=-2\Lambda\beta_{(p,q+1)}\nonumber\\
L\gamma_{(p-1,q)}=J\mu_{(p+1,q)}\quad\quad &,& \quad\quad J\gamma_{(p-1,q)}=-2\Lambda\mu_{(p+1,q)},
\end{eqnarray}
then the bilinear forms $(\psi\overline{\psi})_p$ of $\psi$ satisfy the K\"{a}hlerian CKY equation.
\end{proposition}
\begin{proof}
One can easily see from (36) that if the conditions (67) are satisfied, then the bilinear forms $(\psi\overline{\psi})_p$ satisfy the following equation
\begin{eqnarray}
\nabla_{X_a}(\psi\overline{\psi})_p&=&\frac{1}{p+1}i_{X_a}d(\psi\overline{\psi})_p-\frac{1}{m-p+1}e_a\wedge\delta(\psi\overline{\psi})_p\nonumber\\
&&+\frac{1}{q+1}i_{JX_a}d^c(\psi\overline{\psi})_p-\frac{1}{m-q+1}Je_a\wedge\delta^c(\psi\overline{\psi})_p
\end{eqnarray}
which is the K\"{a}hlerian generalization of the CKY equation given in (33).
\end{proof}

%\references%

\end{document}